\documentclass[11pt]{amsart}
\usepackage{graphicx}
\usepackage{amsmath,amsfonts,amssymb,amsthm,mathtools,pifont,float}

\usepackage{tikz-cd,quiver}

\usepackage[margin=1.5in]{geometry}
\usepackage{array,pbox,multicol,longtable,multirow}

\usepackage{verbatim,fancyvrb,fvextra}
\usepackage{soul} 
\usepackage{MnSymbol} 
\usepackage{hyperref}
\hypersetup{
    colorlinks=true,
    linkcolor=black,
    urlcolor=black,
    citecolor=black
    }

\usepackage[dvipsnames]{xcolor}

\newtheorem{thm}{Theorem}[section]
\newtheorem{lemma}[thm]{Lemma}
\newtheorem{prop}[thm]{Proposition}
\newtheorem{coro}[thm]{Corollary}

\newtheorem{example}[thm]{Example}

\newtheorem{remark}[thm]{Remark}
    
\newcommand{\am}{\operatorname{Am}}
\newcommand{\aut}{\operatorname{Aut}}
\newcommand{\im}{\operatorname{Im}}

\newcommand{\pic}{\operatorname{Pic}}

\newcommand{\pp}{\mathbb{P}}
\newcommand{\zz}{\mathbb{Z}}
\newcommand{\cc}{\mathbb{C}}
\newcommand{\autp}{\operatorname{AutP}}

\numberwithin{equation}{section} 

\author{Shreya Sharma}
\address{Shreya Sharma, University of South Carolina, SC, USA.}
\email{shreyas@email.sc.edu}

\title{Amitsur groups of primitive Fano threefolds}

\begin{document}
\begin{abstract}
    We classify possible Amitsur groups of smooth primitive Fano threefolds defined over $\cc$ that admit a faithful action of a finite group. We also classify the Amitsur groups for toric Fano threefolds.
\end{abstract}

\maketitle

\section{Introduction}
Let $X$ be a proper variety over $\cc$ with a faithful action of a finite group $G$, and $\mathcal{L}$ be a line bundle on $X$. We say $\mathcal{L}$ is $G$-invariant if there are isomorphisms $\phi_g: g^*\mathcal{L}\to \mathcal{L}$ for all $g\in G$. If the isomorphisms $\{\phi_g\}_{g\in G}$ can be chosen such that $\phi_{gh}= \phi_h \circ h^*(\phi_g)$ for all $g,h \in G$, then we say $\mathcal{L}$ is $G$-linearizable. In other words, we say $\mathcal{L}$ is $G$-linearizable if the $G$-action on $X$ lifts to a fiber-wise linear action on the total space of $\mathcal{L}$. A particular choice of this lift is called a linearization of $\mathcal{L}$. 

A morphism of $G$-linearizable line bundles is a morphism of line bundles that is equivariant as a map on their total spaces. We denote the group of isomorphism classes of $G$-invariant line bundles on $X$ by $\pic(X)^G$ and the group of isomorphism classes of $G$-linearizable line bundles with a choice of linearization by $\pic(X,G)$. Clearly, every $G$-linearizable line bundle is also $G$-invariant, so there is a forgetful group homomorphism, $\pic(X,G) \to \pic(X)^G$. 

The \emph{Amitsur subgroup} measures the failure of line bundles to be $G$-linearizable. More precisely, given a line bundle $\mathcal{L}\in \pic(X)^G$, we can construct a cohomology class, $\partial(\mathcal{L})\in H^2(G,\mathbb{C}^\times)$; see \cite[\S A]{BCDP}) that gives the exact sequence of abelian groups 
\begin{equation}    \label{Exactsequence:Amitsur}
1 \rightarrow \operatorname{Hom}(G,\mathbb{C}^\times) \rightarrow \pic(X,G) \xrightarrow{\epsilon} \pic(X)^G \xrightarrow{\partial} H^2(G,\mathbb{C}^\times).
\end{equation}
The Amitsur subgroup is defined as the image of the map $\partial$, 
\[
\am(X,G):= \im(\partial). 
\]
From \eqref{Exactsequence:Amitsur}, we note that this group is isomorphic to $\operatorname{coker}(\pic(X,G) \to \pic(X)^G)$. In this article, we will be mainly concerned with the isomorphism classes of $G$-linearized line bundles rather than the embedding into $H^2(G,\cc^\times)$, therefore, we will refer the Amitsur subgroup as the \emph{Amitsur group}.

The Amitsur group has been shown to be an equivariant birational invariant in \cite{BCDP}, where the authors use it to distinguish conjugacy classes of finite subgroups of $\operatorname{Cr}_3(\cc)$. It provides an obstruction to equivariant versions of rationality, such as linearizability; see for example, \cite[Lemma 2.2]{cheltsov2025equivariantunirationalityfanothreefolds}. Dolgachev determines the Amitsur groups for smooth curves in \cite{dolga-modular-curves}. In \cite[A.7]{BCDP}, the authors classify the possible Amitsur groups for rational surfaces. It is therefore natural to investigate these groups in dimension $3$. 

\medskip
A Fano threefold is imprimitive if it is isomorphic to the blow-up of a Fano threefold along an irreducible smooth curve, otherwise it is \emph{primitive}. By Mori-Mukai classification, the Picard rank $\rho$ of a smooth primitive Fano threefold is at most $3$ \cite{MM81classtable}. There are $17$ deformation families of smooth Fano threefolds with $\rho=1$,  also called \emph{prime} Fano threefolds, and there are $13$ deformation families of smooth primitive Fano threefolds with $\rho>1$.

The main objective of this article is to classify the possible Amitsur groups of smooth primitive Fano threefolds. In \cite[\S 6]{duncan2025numericalamitsurgroup}, the authors provide a method for classifying the Amitsur groups of a smooth projective toric variety; see also \cite{kresch2025equivariantunirationalitytoricvarieties}. As an application, we also give a classification of possible Amitsur groups of smooth toric Fano threefolds.

More precisely, 
\begin{thm}
Let $X$ be a smooth Fano threefold in any of the following families
\begin{center}
    \textnumero 1.1,\; \dots \; \textnumero 1.17,\; \textnumero 2.2,\; \textnumero 2.6,\; \textnumero 2.8,\; \textnumero 2.18,\; \textnumero 2.24,\; \textnumero 2.32,\; \textnumero 2.33,\;  \textnumero 2.34,\;  \textnumero 2.35,\;  \textnumero  2.36,\;  \textnumero 3.1,\;  \textnumero 3.2,\; \textnumero 3.4,\;  \textnumero 3.25,\;  \textnumero 3.26,\; \textnumero 3.27,\;   \textnumero 3.28,\;  \textnumero 3.29,\;  \textnumero 3.30,\; \textnumero  3.31,\;  \textnumero 4.9,\; \textnumero 4.10,\;
   \textnumero 4.11,\; \textnumero 4.12,\; \textnumero 5.2,\; \textnumero 5.3.
\end{center}
Let $G\subseteq\aut(X)$ be a finite group. Then the Tables \ref{table:primitive} and \ref{table:toric} list the largest possible Amitsur group $\am(X,G)$. Moreover, for each entry in tables, there is an $X$ and a group $G$ for which that possibility occurs. 
\end{thm}

This document is organized as follows. In \S 2 we recall basic facts on linearization of line bundles on smooth varieties. An important ingredient in the computation of Amitsur groups is the group $\autp(X)$; we recall in \S 2 its definition and the classification for smooth primitive Fano threefolds. In \S 3 we apply the methods of \cite{duncan2025numericalamitsurgroup} to determine the possible Amitsur groups of smooth toric Fano threefolds. In \S 4 we study the Amitsur groups of smooth Fano threefolds with Picard rank $1$. The remaining sections focus on the Picard rank $\rho =2, 3$. Section 5 treats smooth primitive Fano threefolds that arise as a smooth divisor on fourfolds. In \S 6 we discuss the family \textnumero 3.2, while \S 7 concerns the double covers. Finally, in \S8 we present tables listing the possible Amitsur groups of all smooth primitive and toric Fano threefolds; see Tables~\ref{table:primitive} and~\ref{table:toric}.

\medskip
\noindent \textbf{Acknowledgments.} The author is grateful to Ivan Cheltsov for pointing out that the upper bound in Lemma \ref{Am:3.2} is sharp and for indicating the relevant result in \cite{abban2025kstabilityfano3foldsworld}.

\section{Preliminaries}

\subsection{Notation and Conventions} We work over the field of complex numbers $\cc$. For a finite group $G$, we say $X$ is a $G$-variety if $G$ acts on $X$ by morphisms. A morphism (or, a rational map) $\phi: X \to Y$ of $G$-varieties $X$ and $Y$ is said to be $G$-equivariant if $\phi(g x)=g\phi(x)$ for all $g\in G$ and $x\in X$. 

Let $X$ be a smooth projective $G$-variety. Recall that the Amitsur group is defined as
\[
\am(X,G):= \im \partial \cong \operatorname{coker}(\pic(X,G)\to \pic(X)^G).
\]
For a thorough account on linearization of line bundles, see \cite[\S 2]{duncan2025numericalamitsurgroup}. \medskip

We use the notation $\overline{D}$ to denote the image of a divisor class $D\in \pic(X)^G$ under the map $\pic(X)^G \to \am(X,G)$. A deformation family of smooth Fano threefolds is denoted by \textnumero $\rho$.N, where $\rho$ is the Picard rank and $N$ is the number in the Mori-Mukai classification table (\cite{MM81classtable, MMerratum, fanography}). We will only work with smooth members of these families and assume that the $G$-action is faithful, that is, $G\subseteq \aut(X)$.

In \S \ref{divisorsonP2P2}, we consider Fano threefolds $X$ that appear as divisors in product of projective spaces $Y$.  The following observation gives a criterion for the natural embedding $\iota : X \to Y$ to be $G$-equivariant.

\begin{remark}  \label{remark:equivarince-of-embedding}
  \normalfont Let $Y:=\pp^{n_1}\times \dots \times \pp^{n_r}$, and let $\iota : X \to Y$ denote a smooth projective subvariety. Denote by $\operatorname{pr}_i\colon Y\to\pp^{n_i}$ the
natural projections and set $H_i :=\operatorname{pr}_i^*\mathcal{O}_{\pp^{n_i}}(1)$. Let $G\subseteq \aut(X)$ be a finite group. Suppose that $\operatorname{pr}_i|_X: X \to \pp^{n_i}$ is the morphism induced by the linear system $|\iota^*H_i|$. If the set $\{ \iota^*H_1, \dots , \iota^*H_r\}$ is $G$-invariant, then the $G$-action extends to $Y$. In particular, $\iota$ is $G$-equivariant. 
\end{remark}

\medskip
The following instrumental fact says that the canonical divisor, $K_X$ is $G$-linearizable for any $G$. 
\begin{prop}\cite[Proposition 2.11]{BCDP}   \label{KXlinearizable}
    If $X$ is a smooth $G$-variety, then the canonical bundle has a canonical linearization. 
\end{prop}

As an immediate consequence, we have

\begin{coro}
    Let $G \subset \operatorname{PGL}_{n+1}(\mathbb{C})$ be a finite group. Then $\am(\pp^n,G)$ is isomorphic to a subgroup of $\zz/(n+1)\zz$. 
\end{coro}
\begin{proof}
    Note that $-K_X= (n+1)H$, where $H$ is a representative of the class of a general hyperplane on $\pp^n$. 
\end{proof}

The unique smooth Fano threefold in family \textnumero $1.17$ is isomorphic to $\pp^3$. In the next example, we show that each subgroup of $\zz/4\zz$ can be realized as the Amitsur group for a suitable group $G \subset \operatorname{PGL}_4(\mathbb{C})$. 

\begin{example}   \label{amforP3}
 \normalfont We know that $\operatorname{Pic}(\mathbb{P}^3)= \mathbb{Z}[H]$, where $H$ is the divisor corresponding to $\mathcal{O}_{\mathbb{P}^3}(1)$ and $-K_X = 4H$. Let $G \subset \aut(\pp^3)\cong\operatorname{PGL}_4(\cc)$ be a finite subgroup. Note that $\pic(\pp^3)^G=\pic(\pp^3)$ and $\zz[-K_X] \subseteq \pic(X,G)$ by Proposition \ref{KXlinearizable}. We get
\[  
\am(\pp^3,G) \subseteq \zz/4\zz.
\] 
We now show that all nontrivial possibilities for the Amitsur subgroup of $\mathbb{P}^3$ can be realized for suitable choices of group $G$. Consider the group $\mathcal{G}\cong (\zz/4\zz)^2$ generated by automorphisms of $\pp^3$
\begin{align}    \label{dihedralorder8}
\begin{split}
    \sigma : [x_0:x_1:x_2:x_3] &\mapsto [x_0:i x_1:-x_2:-ix_3]\\
    \tau : [x_0:x_1:x_2:x_3] &\mapsto [x_3:x_0:x_1:x_2].   
\end{split}
\end{align}
Then neither $H$ nor $2H$ is $\mathcal{G}$-linearizable. So, $\am(\mathbb{P}^3, \mathcal{G})\cong\mathbb{Z}/4\mathbb{Z}$.
Now consider the group $\mathcal{G}'\cong (\zz/2\zz)^2$ generated by 
\begin{align}   \label{c4c4semidirectc4}
\begin{split}
    \sigma' : [x_0:x_1:x_2:x_3] &\mapsto [x_0:- x_1:x_2:-x_3]\\
    \tau' : [x_0:x_1:x_2:x_3] &\mapsto [x_1:x_0:x_3:x_2]. 
\end{split}
\end{align}
Then $H$ is not $\mathcal{G}'$-linearizable, but $2H$ is. So, $\am(\mathbb{P}^3, \mathcal{G}')\cong\mathbb{Z}/2\mathbb{Z}$.
\end{example}

The following result says that $\am(X,G)$ is an equivariant birational invariant (\cite[Theorem A.1]{BCDP}).
\begin{thm} \label{imprimitivebcdp}
    If $X$ and $Y$ are smooth projective $G$-varieties that are $G$-equivariantly birationally equivalent, then 
    \[
    \am(X,G)=\am(Y,G).
    \]
    In particular, if $X$ is the blow-up of $Y$ in a $G$-invariant subvariety and $E$ is the exceptional divisor, then $E$ is $G$-linearizable. 
\end{thm}

The following observation is immediate, see also \cite[Lemma 2.1]{kreschtschinkel-unramified}.
\begin{prop}  \label{g-eq-divisors}
    Let $f:X \to Y$ be a $G$-equivariant morphism such that $f^*: \pic(Y) \to \pic(X)$ is an isomorphism. Then $\am(X,G) = \am(Y,G)$.  
\end{prop}
\begin{proof}
It follows that $f^* : \operatorname{Pic}(X)^G \xrightarrow[]{\sim} \operatorname{Pic}(Y)^G$. The rest follows from the commutativity of the following diagram. 
\[\begin{tikzcd}
	1 & {\operatorname{Ker}\epsilon_X} & {\operatorname{Pic}(X,G)} & {\operatorname{Pic}(X)^G} & {\operatorname{Am}(X,G)} & 1 \\
	1 & {\operatorname{Ker}\epsilon_Y} & {\operatorname{Pic}(Y,G)} & {\operatorname{Pic}(Y)^G} & {\operatorname{Am}(Y,G)} & 1
	\arrow[from=1-1, to=1-2]
	\arrow[from=1-2, to=1-3]
	\arrow["{\epsilon_X}", from=1-3, to=1-4]
	\arrow["{\partial_X}", from=1-4, to=1-5]
	\arrow[from=1-5, to=1-6]
	\arrow[from=2-1, to=2-2]
	\arrow[from=2-2, to=1-2]
	\arrow[from=2-2, to=2-3]
	\arrow["{f^*}", from=2-3, to=1-3]
	\arrow["{\epsilon_Y}", from=2-3, to=2-4]
	\arrow["{f^*}", from=2-4, to=1-4]
	\arrow["{\partial_Y}", from=2-4, to=2-5]
	\arrow[from=2-5, to=1-5]
	\arrow[from=2-5, to=2-6]
\end{tikzcd}\]
\end{proof}

For a $\mathcal{L}\in \pic(X)^G$, its Amitsur period is the order of $\partial(\mathcal{L})$ in $H^2(G,\cc^\times)$ \cite[\S 2]{duncan2025numericalamitsurgroup}. We denote it by $m$ in the following proposition.

\begin{prop} \label{Gsubrepn_result}
Let $\mathcal{L}$ be a $G$-invariant line bundle on $X$.
   If there exists a $G_{\mathcal{L}}$-subrepresentation of $H^0(X,\mathcal{L})$ of dimension $d$, then $d\partial(\mathcal{L})=0$. In particular, $m$ divides $d$.
\end{prop}
\begin{proof}
   For a proof, see \cite[Proposition 2.6]{duncan2025numericalamitsurgroup}.  
\end{proof}

\subsection{The group \texorpdfstring{$\autp(X)$}{autp}} 
The natural action of the automorphism group of a Fano variety on its Picard group can provide insight into the structure of the automorphism group itself. For a smooth projective variety $X$, consider the homomorphism 
\[
\aut(X) \rightarrow \aut(\pic(X)).
\]
Let $G\subseteq \aut(X)$ be a finite group. We write $\autp(X,G)$ for the image of $G$, and the image of the entire automorphism group $\aut(X)$ will be denoted by $\autp(X)$. 

The classification of the group $\autp(X)$ for smooth Fano threefolds was studied by the author in \cite{sharma2026actionspicardgroupsmooth}. In particular, the classification is complete for both smooth primitive Fano threefolds and smooth toric Fano threefolds. Throughout this article, we make use of this classification, which is summarized in Tables~\ref{table:primitive} and~\ref{table:toric}.


\section{Toric Fano threefolds}
There are $18$ deformation families of smooth toric Fano threefolds
\begin{center}
   \textnumero 1.17,\; \textnumero 2.33,\;  \textnumero 2.34,\;  \textnumero 2.35,\;  \textnumero  2.36,\;  \textnumero 3.25,\;  \textnumero 3.26,\; \textnumero 3.27,\;   \textnumero 3.28,\;  \textnumero 3.29,\;  \textnumero 3.30,\; \textnumero  3.31,\;
   \textnumero 4.9,\; \textnumero 4.10,\;
   \textnumero 4.11,\; \textnumero 4.12,\; \textnumero 5.2,\; \textnumero 5.3.
\end{center}
Let $X$ denote a member. In \cite[\S 6]{duncan2025numericalamitsurgroup}, the Amitsur group of toric varieties was discussed; we begin by recalling the main results from that section. 

Suppose $T$ is the torus of a smooth projective toric variety $X$ and $\Sigma$ is the fan. The group of torus-invariant divisor classes of $X$ is $\operatorname{TDiv}(X)\cong \zz^{\Sigma(1)}$.

\begin{prop}  \label{amfor_smoothtoric}
    For every subgroup $J$ of $\autp(X)$ and its preimage $W$ in $\aut(\Sigma)$, there exists a finite subgroup $G$ of $\aut(X)$ such that $\autp(X,G)=J$, and
    \begin{equation}  \label{Am-for-toric-iso}
         \operatorname{coker}(\operatorname{TDiv}(X)^W \to \pic(X)^J) \xrightarrow{\sim} \am(X,G)
    \end{equation}
    is an isomorphism.
\end{prop}
\begin{proof}
    For a proof, see Lemma 6.3, Theorem 6.4, and Proposition 6.5 in \cite{duncan2025numericalamitsurgroup}.
\end{proof}
The groups $J$ are completely determined by the fan; combined with the above proposition, this yields a complete classification of the Amitsur groups of $X$. We illustrate this with an example of an imprimitive Fano threefold $X$ in family \textnumero 5.2. 

\begin{example}
\normalfont Let $X$ be a Fano threefold in family \textnumero 5.2. Suppose that the Cox ring of $X$ has coordinates  with weights 
\begin{center}
\begin{tabular}{c c c c c c c c}
   $x_0$  & $x_1$ & $x_2$ & $x_3$ & $x_4$ & $x_5$ & $x_6$ & $x_7$ \\
   \hline
   1 & -1 & 1 &  0 & 0 & 0 &  0 &  0      \\
   0 & 0 & 0 &  1 &  -1 & 1 & 0 & 0          \\
   0 & 0 & 1 &  -1 & 1 & 0 &  0 &  0         \\
   -1 & 0 & 0 &  0 & 0 & 0 & 1 &  1        \\
   0 & 1 & -1 & 1 & 0 & 0 & 0 & 0
\end{tabular}
\end{center}
Then $\pic(X)= \zz[D_{x_0}, D_{x_1}, D_{x_4}, D_{x_5}, D_{x_6}]$. Recall from Table \ref{table:toric} that $\autp(X)=\zz/2\zz$. So, we have two possibilities: $J$ is isomorphic to $\zz/2\zz$ or it is trivial. The automorphism group of the fan $\Sigma$ is 
\[
\aut(\Sigma) = \langle (1\;5)(2\;4)(6\;7), (6\;7)\rangle \cong (\zz/2\zz)^2, 
\]
where the $2$-cycle $(i\;j)$ swaps the divisors classes $D_{x_i}$ and $D_{x_j}$.

First suppose that $J=0$. Then $\pic(X)^J=\pic(X)$ and $W=\{(1), (6\;7)  \}$
\[
\operatorname{TDiv}(X)^W = \zz[D_{x_0}, D_{x_1}, D_{x_2}, D_{x_3}, D_{x_4}, D_{x_5}, D_{x_6}+D_{x_7}].
\]
We find that $\im(\operatorname{TDiv}(X)^W \to \pic(X))\cong \pic(X)$, thus $\am(X,G)=0$ for some finite $G\subseteq \aut(X)$. 

If $J\cong \zz/2\zz$, then $W= \{(1\;5)(2\;4)(6\;7), (1\;5)(2\;4)\}$, and 
\begin{align*}
    \pic(X)^J &= \zz[D_{x_0}, D_{x_1}+D_{x_5}, D_{x_4}+D_{x_5}, D_{x_6}] \cong \zz^4 \\
    \operatorname{TDiv}(X)^W &= \zz[D_{x_0}, D_{x_1}+D_{x_5}, D_{x_2}+D_{x_4}, D_{x_3}, D_{x_6}+D_{x_7}].
\end{align*}
We find that the image of the linear map $\operatorname{TDiv}(X)^W \to \pic(X)^J$ is isomorphic to
\[
\zz[D_{x_0}, D_{x_1}+D_{x_5}, 2D_{x_4}+D_{x_5}-D_{x_1}, D_{x_0}+D_{x_1}-D_{x_4}+D_{x_6}, 2D_{x_6}],
\]
and the cokernel is isomorphic to $\zz/2\zz$. Thus, $\am(X,G)\cong \zz/2\zz$ for some finite group $G\subseteq \aut(X)$. 
\end{example}

The Amitsur groups of the remaining families admitting a toric representative are computed in a similar manner. The possible Amitsur groups are listed in Table~\ref{table:toric}.


\section{\texorpdfstring{Prime Fano threefolds }{Rank 1}}    \label{Amrank1}
There are $17$ deformation families of smooth prime Fano threefolds, that is, Fano threefolds with Picard rank $1$
\begin{center}
    \textnumero 1.1,\;  \textnumero 1.2,\dots , \textnumero 1.17.
\end{center}
Let us denote a member of these families by $X$. Recall that the index $r$ is the maximal positive integer such that the anticanonical divisor $-K_X = rH$ for some ample divisor $H$, called the fundamental divisor (\cite{isko1}).

\begin{prop}
    If $X$ is a smooth Fano variety of index $r$ with $\operatorname{rk}\pic(X)^G =1$, then $\am(X,G)\subseteq \zz/r\zz$ for all finite groups $G\subseteq \aut(X)$.
\end{prop}
\begin{proof}
   Let $G\subseteq \aut(X)$ be a finite group. We know from Proposition \ref{KXlinearizable} that $-K_X=rH$ is $G$-linearizable. Then $\am(X,G):= \im\partial$ is a subgroup of the cyclic group of order $r$ generated by $\partial(H)$, equivalently, $\am(X,G)\subseteq \mathbb{Z}/r\mathbb{Z}$.
\end{proof}

In particular, $H$ is $G$-linearizable for $r=1$. The classification of Amitsur groups for Picard rank $\rho=1$ and index $1$ follows immediately.

\begin{coro}\label{amfor_index=rank=1}
If $X$ is a smooth Fano threefold in any of the following families 
\begin{center}
    \textnumero 1.1,\; \dots , \textnumero 1.10,
\end{center}
then $\am(X,G)=0$ for all finite subgroups $G\subseteq \aut(X)$.
\end{coro} 

We now focus on the remaining families
\begin{center}
    \textnumero 1.11,\; \dots , \textnumero 1.17.
\end{center}
Recall that $r=3$ if $X$ is in \textnumero 1.16 and $r=4$ if it is in \textnumero 1.17, otherwise $r=2$. The linear system $|H|$ is very ample for the families \textnumero $1.13$, \textnumero $1.14$, and \textnumero $1.15$, and there is an embedding  
\[
\phi_{|H|} : X = X_d \hookrightarrow \mathbb{P}^{d+1}, 
\]
where $d=H^3$.
For the family \textnumero $1.12$, $d=2$ and $\phi_{|H|}: X \to \mathbb{P}^{3}$ is the double covering. If $X$ is in \textnumero $1.11$, $\phi_{|H|}$ is only a rational map to $\mathbb{P}^2$ (\cite[Theorem 4.2]{isko1}). 

\vspace{0.2cm}
Suppose that $X$ is a smooth Fano threefold in  family \textnumero $1.11$, then it can be described as a hypersurface of degree $6$ in $\mathbb{P}(1,1,1,2,3)$ (\cite{isko2,MM81classtable}). Let $x_0,x_1,x_2,x_3,x_4$ be the coordinates on the weighted projective space $\mathbb{P}(1,1,1,2,3)$ having weights $1,1,1,2,$ and $3$, respectively. Recall that $-K_X=2H$ where $H$ is the fundamental divisor.

The next lemma shows that $H$ is $G$-linearizable for any finite group $G\subseteq \aut(X)$.

\begin{lemma}    \label{Am:1.11}
    Let $X$ be a Fano threefold in family \textnumero 1.11. Let $G\subseteq \aut(X)$ be a finite group. Then $\am(X,G)=0$.
\end{lemma}
\begin{proof}
Consider the map induced by $|3H|$, $\mathbb{\phi}_{|3H|}: X \rightarrow \mathbb{P}^{13}=\pp(H^0(X, 3H)^{\vee})$, which is obtained by restriction of the map $\pp \to \pp^{13}$ 
\begin{align*}
             [x_0:x_1:x_2:x_3:x_4] \to [x_0^3: \dots : x_2x_3 : x_4]
\end{align*}
to $X$, where $\pp:=\pp(1,1,1,2,3)$.
An automorphism $\sigma$ of $ \pp$ is
\begin{multline*}
   [x_0:x_1:x_2:x_3:x_4] \mapsto [a_{00}\; x_0  + a_{01}\; x_1 + a_{02}\; x_2:\; a_{10}\; x_0  + a_{11}\; x_1 + a_{12}\; x_2 \;: \\a_{20}\; x_0  + a_{21}\; x_1 + a_{22}\; x_2 : b x_3 + A_2(x_0,x_1,x_2) : c x_4 + B_3(x_0,x_1,x_2,x_3)],
\end{multline*}
where $b,c,a_{ij} \in \mathbb{C}$ for $0 \leq i,j \leq 2$, $A_2$ is a homogeneous polynomial of degree $2$ in $x_0,x_1,x_2$ and $B_3$ is a homogeneous polynomial of degree $3$ in $x_0,x_1,x_2, x_3$. 
We have the decomposition $\aut(\pp)\cong\aut(\pp)^\circ= R_u \rtimes M$, where $R_u$ is the unipotent radical and $M:=\{ \sigma \;|\; A_2=B_3=0\}$ is a maximal reductive (Levi) subgroup of $\aut(\pp)$. If $G\subseteq \aut(\pp)$ is a finite group, then it is conjugate to a subgroup of $M$. 
Assume that $c=1$, without loss of generality.

Consider the action on $\pp^{13}$ induced by $G$ via $\mathbb{\phi}_{|3H|}$, it fixes the point
\begin{align*}
p:= \phi_{|3H|}([0:0:0:0:1])= [0:\dots :0:1] \in \pp^{13},
\end{align*}
and consequently, there is a $G_{3H}$-subrepresentation of $H^0(X, 3H)^{\vee}$ of degree $1$. 
It now follows from Proposition \ref{Gsubrepn_result} that $\partial(3H)=0$.  Since $-K_X=2H$, $\partial(2H)=0$. Thus, $H$ is $G$-linearizable, and
$\am(X,G)=0$ for any finite subgroup $G\subset \aut(X)$. 
\end{proof}

A smooth Fano threefold $X$
in family \textnumero 1.12 is a double cover of $\pp^3$ branched along a smooth quartic surface $B$ (\cite{isko2,MM81classtable}). Let $\phi:=\phi_{|H|}: X \to \pp^3$ denote the double covering, where the divisor class $[H]$ is also the pullback to $X$ of the class of a general hyperplane $L$ on $\pp^3$. 
The index is $r= 2$; therefore, we have 
\[\am(X,G)\subseteq \zz/2\zz.
\]

In the following lemma, we show that $\am(X,G)\cong \zz/2\zz$ for the choice of a finite group $G$ and a smooth quartic $B$ defining the branch locus of $X$.

\begin{lemma}   \label{Am:1.12}
Let $X$ be a Fano threefold in family \textnumero 1.12. Let $G\subseteq \aut(X)$ be a finite group. Then $\am(X,G)$ is isomorphic to $\zz/2\zz$ or trivial.
\end{lemma}
\begin{proof}
Consider the group $\mathcal{G}\cong (\zz/4\zz)^2$ generated by automorphisms of $\pp^3$ from \eqref{dihedralorder8}
\begin{equation*}      
\begin{split}
   \sigma :[x_0:x_1:x_2:x_3] &\mapsto [x_0:ix_1:-x_2:-ix_3]\\
    \tau : [x_0:x_1:x_2:x_3] &\mapsto [x_3:x_0: x_1:x_2], \end{split}
\end{equation*}
where $x_0,x_1,x_2$, and $x_3$ are coordinates of $\mathbb{P}^3$. Suppose that $B$ is given by the equation
\[
x_0^4 + x_1^4 + x_2^4 + x_3^4 =0,
\]
it is a $\mathcal{G}$-invariant and can be verified to be smooth and irreducible quartic surface in $\pp^3$. 
So, $\mathcal{G} \subseteq \aut(\pp^3;B)$. Let $\bar{\mathcal{G}}\subset \aut(X)$ be the preimage of $\mathcal{G}$.

From Example \ref{amforP3}, we know that both $L$ and $2L$ are not $\mathcal{G}$-linearizable, so $\am(\pp^3,\mathcal{G})\cong \zz/4\zz$. Furthermore, we observe that $\am(\pp^3, \mathcal{G})\rightarrow \am(\pp^3,\bar{\mathcal{G}})$ is a surjection with kernel generated by the class of $B/2=2L$. Together with the Proposition \ref{g-eq-divisors}, we get
\[
\am(X,\bar{\mathcal{G}})\cong \am(\pp^3,\bar{\mathcal{G}}) \cong  \zz/2\zz.
\]
\end{proof}

The following observation is useful for determining the linearizability of line bundles on hypersurfaces.

\begin{prop}   \label{duncan-dandn+1-coprime}
Let $X$ be a smooth hypersurface of degree $d$ and $\mathcal{L}$ is a line bundle on $X$ such that $\phi_{|\mathcal{L}|} : X \to \pp^n$ is an embedding. Suppose that the integers $d$ and $n+1$ are coprime. Let $G$ be a finite subgroup acting faithfully on $\mathbb{P}^n$ such that $X$ is $G$-invariant. Then $\mathcal{L}$ is $G$-linearizable. 
\end{prop}
We refer to \cite[Lemma 2.4]{abban2025kstabilityfano3foldsworld} for a proof.

\medskip

Consider families \textnumero 1.13 and \textnumero 1.16. In both cases, the divisor class corresponding to the pullback line bundle $\mathcal{O}_{\pp^4}(1)|_X$ is $[H]$ and there is a canonical embedding $\phi_{|H|}: X \hookrightarrow \pp^4$. 

\begin{coro}  \label{Am:1.13,1.16}
    Let $X$ be a smooth Fano threefold in \textnumero $1.13$ or \textnumero $1.16$. Let $G\subseteq \aut(X)$ be a finite group. Then $\am(X,G)=0$.
\end{coro}
\begin{proof}
If $X$ is in \textnumero 1.13, it is a hypersurface of degree $3$ and if it is in \textnumero 1.16, it is a hypersurface of degree $2$. Note that automorphisms of $X$ are obtained from those of $\mathbb{P}^4$ by restricting them to $X$. So, if $G\subseteq \aut(X)$ is a finite group, we may assume that $G\subset \aut(\pp^4)$. Thus, $H$ is $G$-linearizable by the Proposition \ref{duncan-dandn+1-coprime}, and consequently, $\am(X,G)=0$.  
\end{proof}

Suppose now that $X$ is a smooth Fano threefold in \textnumero 1.14 or \textnumero 1.15, then it has index $r=2$; so, $-K_X =2H$ and we have $
\am(X,G) \subseteq \mathbb{Z}/2\mathbb{Z}$. 
We now verify whether there is a finite group $G\subseteq \aut(X)$ such that $H$ is $G$-linearizable.

\begin{lemma}  \label{Am:1.14}
    Let $X$ be a smooth Fano threefold in family \textnumero $1.14$. Let $G\subseteq \aut(X)$ be a finite group. Then $\am(X,G)$ is isomorphic to $\zz/2\zz$ or trivial.
\end{lemma}
\begin{proof}
If $X$ is in \textnumero $1.14$, it is a complete intersection of two quadrics $Q_1$ and $Q_2$ in $\mathbb{P}^5$. Note that $[\mathcal{O}_{\pp^5}(1)]=[H]$ and recall that it induces $\phi_{|H|}: X \hookrightarrow \pp^5$. Let $\sigma$, $\tau $ be automorphisms of $\cc^6$ defined as
\begin{align}
\begin{split}
    \sigma : (x_0:x_1:x_2:x_3:x_4:x_5) &\mapsto (-x_0:-x_1:-x_2:x_3:x_4:x_5)\\
    \tau : (x_0:x_1:x_2:x_3:x_4:x_5) &\mapsto (\zeta x_3:\zeta x_4:\zeta x_5: x_0:x_1:x_2),  
\end{split}
\end{align}
where $x_0, \dots ,x_5$ are the coordinates of $\cc^6$ and $\zeta$ is a primitive sixth root of unity. Denote their images in $\operatorname{PGL}_6(\cc)$ by $\bar{\sigma}$ and $\bar{\tau}$ and let $\mathcal{G}:=\langle \bar{\sigma} , \bar{\tau}\rangle \cong (\zz/2\zz)^2$.
Consider the quadrics
\begin{align} 
\begin{split}
    Q_1 &: a_0x_0^2 + a_1x_1^2 +  a_2x_2^2 +  a_3 x_3^2 +  a_4 x_4^2 +  a_5 x_5^2=0\\
    Q_2 &: a_3 x_0^2 +  a_4 x_1^2 + a_5 x_2^2 + \zeta^2(a_0 x_3^2 + a_1 x_4^2 + a_2 x_5^2)=0.  
\end{split}
\end{align}
where $a_i \in \mathbb{C}$ for $0 \leq i \leq 5$ are distinct. Observe that the complete intersection $Q_1\cap Q_2$ is $\mathcal{G}$-invariant: $\overline{\sigma}$ fixes each quadric, $\overline{\tau}$ preserves the pencil generated by $Q_1$ and $Q_2$. The smoothness and irreducibility of $Q_1\cap Q_2$ can be verified using Magma (\cite{magma}); for example, one may choose 
$(a_0,a_1,a_2,a_3,a_4,a_5)=(1,-1,2,-2,3,-3)$. 
  However, the $\mathcal{G}$-action on $\pp^5$ does not lift to a homomorphism into $\operatorname{GL}_6(\mathbb{C})$. It follows that $H$ is not $\mathcal{G}$-linearizable, and hence, $\am(X,\mathcal{G})\cong\mathbb{Z}/2\mathbb{Z}$.
\end{proof}

The unique smooth Fano threefold $X$ in \textnumero $1.15$ is a section of Pl\"{u}cker embedding of the Grassmannian $\operatorname{Gr}(2,5)$ by codimension $3$ subspace \cite{isko1, MM81classtable}.

\begin{lemma}  \label{Am:1.15}
    Let $X$ be a smooth Fano threefold in family \textnumero $1.15$. Let $G\subseteq \aut(X)$ be a finite group. Then we have $\am(X,G)= 0$.    
\end{lemma}
\begin{proof}
Let $G\subset \aut(X)$ be a finite group. 
Recall that there is an embedding $ \phi_{|H|}:X \hookrightarrow \mathbb{P}^6$ induced by $|H|$, so we may regard $G$ as a subgroup of $\aut(\mathbb{P}^6)\cong \operatorname{PGL}_7(\mathbb{C})$. This implies that there is a projective representation of $G$ of dimension $7$. If $\widetilde{G}\subset \operatorname{SL}_7(\cc)$ is a subgroup of the lifting group of $G$ such that there is a short exact sequence 
\[
1 \to \zz/7\zz \to \widetilde{G} \to G \to 1,
\]
see \cite[\S 2.1 (The lifting group)]{duncan2025numericalamitsurgroup}, then $\widetilde{G}$ has a linear representation of dimension $7$. It follows from Proposition \ref{Gsubrepn_result} that the Amitsur period $m$ of $H$ divides $7$. Since $-K_X=2H$, $m$ also divides $2$. Therefore, $H$ must be $G$-linearizable, and $\am(X,G)=0$. 
\end{proof}
 

\section{\texorpdfstring{Divisors on $\pp^2\times\pp^2$}{Divisors in P2P2}}   \label{divisorsonP2P2}
Let $X$ be a smooth Fano threefold belonging to one of the following deformation families:
\begin{center}
    \textnumero 2.6 (a),\; \textnumero 2.24,\; \textnumero 2.32.
\end{center}
Recall that, by Mori-Mukai classification, a Fano threefold in ~\textnumero 2.6 admits two descriptions: as a divisor (\textnumero 2.6\;(a)) or as a double cover (\textnumero 2.6\;(b)). 
In this section, we consider only the divisor description. The case of Family~\textnumero 2.6(b) will be treated separately in \S\ref{Am:2.6(b)discussion}.

If $X$ is a member of \textnumero 2.6 (a), \textnumero 2.24, or \textnumero 2.32, it can be described as a divisor of bidegree $(n,m)$ on $\pp^2\times \pp^2$ (\cite[Table 2]{MM81classtable}). Let $\iota : X \hookrightarrow \mathbb{P}^2 \times \mathbb{P}^2$ be the embedding of $X$. 
Let $H_1$ and $H_2$ denote the pullbacks to $\pp^2\times\pp^2$ of the hyperplane classes from the first and second factors of $\pp^2\times \pp^2$, respectively. Then $X \in |nH_1+mH_2|$ and $\operatorname{Pic}(X)= \zz [\iota^*H_1,\iota^*H_2] $. 

\smallskip 

We also recall the following about the action of a finite group $G\subseteq \aut(X)$ on the Picard group for these families (Table \ref{table:primitive})
\[
\autp(X) = \begin{cases}
     \zz/2\zz  & \text{\textnumero }2.6(a), \;\text{\textnumero }2.32 \\
     0  & \text{\textnumero } 2.24.
\end{cases}
\]
By Remark \ref{remark:equivarince-of-embedding}, we can assume that $G$ is a subgroup of $\aut(\mathbb{P}^2\times \mathbb{P}^2)$ and $\iota$ is $G$-equivariant.

\begin{remark}   \label{remark:equivariantprojections}
\normalfont    Suppose that $\autp(X,G)=0$. Then $\pic(X)^G= \zz[\iota^*H_1,\iota^*H_2]$ and the projection maps $\pp^2\times\pp^2 \to \pp^2$ are $G$-equivariant. Therefore, the pullbacks $3\iota^*H_i$ of canonical divisors $-K_{\pp^2}$ to $X$ are both $G$-linearizable, $i=1,2$.
\end{remark}
Let $\overline{\iota^*H_i}$ denote the image of $\iota^*H_i$ under the map $\pic(X)^G \to \am(X,G)$ for $i=1,2$. \smallskip

In Lemma \ref{Am:2.6b}, we will show that $\am(X,G)\subseteq \zz/3\zz$ for any finite group $G\subseteq \aut(X)$, where $X$ is in family \textnumero 2.6(b). We now prove that the analogous holds for family~\textnumero 2.6(a).
\begin{lemma}    \label{Am:2.6a}
    Let $X$ be a Fano threefold in \textnumero 2.6(a). Let $G\subseteq \aut(X)$ be a finite group. Then $\am(X,G)$ is isomorphic to $\zz/3\zz$ or trivial.
\end{lemma}
\begin{proof}
 A Fano threefold $X$ in family \textnumero 2.6~(a) is a divisor of bidegree $(2,2)$ on $\pp^2\times\pp^2$ with anticanonical divisor $-K_X = \iota^*H_1 + \iota^*H_2$. We first establish an upper bound on $\am(X,G)$.

Let $G \subseteq \aut(X)$ be a finite group. First suppose that $\autp(X,G)=\zz/2\zz$. Then $\pic(X)^G=\zz[-K_X]$. By Proposition \ref{KXlinearizable}, $-K_X$ is $G$-linearizable for any $G\subset \aut(X)$. Thus, $\am(X,G)=0$.

Now suppose that $\autp(X,G)=0$. Using Remark \ref{remark:equivariantprojections}, we get 
\begin{equation}   \label{relation-for-2.6(a)}
\overline{\iota^*H_1} =2\overline{\iota^*H_2},\; \overline{\iota^*H_2}=2\overline{\iota^*H_1}.
\end{equation}
Therefore, $\am(X,G) \subseteq \zz/3\zz$. 

To complete the proof, it suffices to show that there is Fano threefold $X$ in \textnumero 2.6 and a finite group $G$ such that the upper bound on $\am(X,G)$ is sharp. We do this by constructing a member of \textnumero 2.6~(a).

Consider $\mathcal{G}\cong (\zz/3\zz)^2$ generated by automorphisms
\begin{align}   \label{2-6agroup}
\begin{split}
\sigma : ([x_0:x_1:x_2],[y_0:y_1:y_2]) &\mapsto ([x_2:x_0:x_1],[y_0:\omega y_1:\omega^2 y_2]) \\
\tau  : ([x_0:x_1:x_2],[y_0:y_1:y_2]) &\mapsto ([x_0:\omega x_1:\omega^2 x_2],[y_2:y_0:y_1]), 
\end{split}
\end{align}
where $([x_0:x_1:x_2],[y_0:y_1:y_2])$ are projective coordinates on $\pp^2\times\pp^2$ and $\omega$ is a primitive cube root of unity.
Now consider the hypersurface defined by the equation
\begin{equation}  \label{eqn-for-2.6(a)}
\begin{aligned}
& x_0^2(y_0^2 + y_1^2 + y_2^2 + y_0y_1 + y_1y_2 + y_0y_2) + \\
& x_1^2(y_0^2 + \omega y_1^2 + \omega^2 y_2^2 + \omega^2 y_0y_1 +  y_1y_2 + \omega y_0y_2) +\\
& x_2^2(y_0^2 + \omega^2 y_1^2 + \omega y_2^2 + \omega y_0y_1 + y_1y_2 
 + \omega^2 y_0y_2) + \\
& x_0x_1(y_0^2 + \omega^2 y_1^2 + 
\omega  y_2^2 + \omega y_0 y_1 + y_1 y_2 + \omega^2 y_0y_2) + \\
& x_1x_2(y_0^2 + y_1^2 + y_2^2 +
y_0y_1 + y_1y_2 +
y_0y_2) + \\
& x_0x_2(y_0^2 + \omega y_1^2 + \omega^2 y_2^2 + \omega^2 y_0y_1 + y_1y_2 + \omega y_0y_2)=0.
\end{aligned}
\end{equation} 
Then this hypersurface is $\mathcal{G}$-invariant. Using Magma \cite{magma}, we verify that it is both smooth and irreducible. Therefore, the equation \eqref{eqn-for-2.6(a)} defines a member of \textnumero 2.6(a). Observe that $\autp(X,\mathcal{G})= 0$. 

Finally, since the $\mathcal{G}$-action does not lift to a linear representation to $\operatorname{GL}_3(\mathbb{C})^{2}$, the divisors $\iota^*H_1$ and $\iota^*H_2$ are not $\mathcal{G}$-linearizable. Thus, we have $\am(X,\mathcal{G}) \cong \mathbb{Z}/3\mathbb{Z}$. 
\end{proof}

Suppose that $X$ is Fano threefold in \textnumero 2.24, then Mori-Mukai describe it as a divisor of bidegree $(n,m)=(1,2)$ in $\pp^2\times\pp^2$. The anticanonical divisor is $-K_X= 2\iota^*H_1+\iota^*H_2$. We know that $\autp(X)=0$ (Table \ref{table:primitive}).

\begin{lemma}  \label{Am:2.24}
     Let $X$ be a Fano threefold in \textnumero 2.24 and $G\subseteq \aut(X)$ be a finite group. Then $\am(X,G)$ is isomorphic to $\zz/3\zz$ or trivial.
\end{lemma}
\begin{proof}
Let $G\subseteq \aut(X)$ be a finite group. Since $-K_X,\; 3\iota^*H_1$, and $3\iota^*H_2$ are $G$-linearizable (Proposition \ref{KXlinearizable}, Remark \ref{remark:equivariantprojections}), we have 
\[
\overline{\iota^*H_1} =\overline{\iota^*H_2},
\]
and therefore, $\am(X,G)$ is isomorphic to a subgroup of the cyclic group $\zz/3\zz$. 

It is possible to realize $\am(X,G)\cong \zz/3\zz$ for a Fano threefold $X$ and some choice of $G\subseteq \aut(X)$. Define $\sigma, \tau \in \aut(\pp^2\times\pp^2)$ by
\begin{align*}
\sigma : ([x_0:x_1:x_2],[y_0:y_1:y_2])  &\mapsto ([x_0:\omega x_1:\omega^2 x_2],[y_0:\omega y_1:\omega^2 y_2])\\
\tau : ([x_0:x_1:x_2],[y_0:y_1:y_2])  &\mapsto ([x_1:x_2:x_0],[y_1:y_2:y_0]),
\end{align*}
where $\omega$ be a primitive cube root of unity. Denote by $\mathcal{G}\cong (\zz/3\zz)^2$ the group generated by $\sigma$ and $\tau$.
Then the hypersurface defined by the equation
\[
x_0y_0^2 + x_1y_1^2 + x_2y_2^2 + x_0y_1y_2 +x_1y_0y_2 +x_2y_0y_1=0
\]
is $\mathcal{G}$-invariant and can be verified to be both smooth and irreducible (\cite{magma}), thus it defines a member of family \textnumero 2.24.
Since the $\mathcal{G}$-action on $\pp^2\times\pp^2$ does not lift to a linear representation, the divisor $\iota^*H_1$ is not $\mathcal{G}$-linearizable. So, $\am(X,\mathcal{G}) \cong \mathbb{Z}/3\mathbb{Z}$. 
\end{proof}

\vspace{0.2cm}

A Fano threefold $X$ in family \textnumero 2.32 is a divisor of bidegree $(n,m)=(1,1)$ on $\pp^2\times\pp^2$. The anticanonical divisor is $-K_X= 2\iota^*H_1+2\iota^*H_2$.

\begin{lemma}  \label{Am:2.32}
     Let $X$ be a Fano threefold in \textnumero 2.32 and $G\subseteq \aut(X)$ be a finite group. Then $\am(X,G)$ is isomorphic to $\zz/3\zz$ or trivial.
\end{lemma}
\begin{proof}
Suppose first that $\autp(X,G)=\zz/2\zz$, then $\pic(X)^G= \zz[\iota^*H_1+\iota^*H_2]$. We observe that $\iota^*H_1+\iota^*H_2$ is $G$-linearizable from 
\[
\iota^*H_1+\iota^*H_2 = \iota^*(-K_{\pp^2\times\pp^2})+K_X
\]
and Proposition \ref{KXlinearizable}. Thus, $\am(X,G)=0$. 

If $\autp(X,G)=0$, then using Remark \ref{relation-for-2.6(a)}, we get
\[
\overline{\iota^*H_1} =2\overline{\iota^*H_2},\; \overline{\iota^*H_2}=2\overline{\iota^*H_1},
\]
in $\am(X,G)$. Therefore, $\am(X,G)$ is isomorphic to a subgroup of $\zz/3\zz$.

Consider the group $\mathcal{G}$ generated by automorphisms in \eqref{2-6agroup}, and the hypersurface 
\[
(x_0y_0 + x_1y_0 + x_2y_0 + x_0y_1 + x_0y_2) + \omega(x_1y_2+x_2y_1) + \omega^2(x_1y_1+x_2y_2) =0
\]
This hypersurface is $\mathcal{G}$-invariant. Moreover, a computation using Magma (\cite{magma}) shows that it is smooth and irreducible, and therefore the hypersurface belongs to family \textnumero 2.32. Now since the $\mathcal{G}$-action does not lift to a group homomorphism to $\operatorname{GL}_3(\cc)^{2}$, $\iota^*H_1$ and $\iota^*H_2$ are not $\mathcal{G}$-linearizable. So, we get $\am(X,\mathcal{G}) \cong \zz/3\zz$. 
\end{proof}

\section{\texorpdfstring{Family \textnumero 3.2}{Family 3.2}}
Let $X$ denote a Fano threefold in the family \textnumero $3.2$. Let $\mathcal{Q}$ be the $\pp^2$-bundle
\[
\pp(\mathcal{O}_S\oplus \mathcal{O}_S(-1,-1)^{\oplus 2})
\]
over $S:=\pp^1 \times \pp^1$. Let $f:\mathcal{Q}\to S$ be the natural projection. Let $\mathcal{L}$ denote the tautological line bundle on $\mathcal{Q}$. Then Mori-Mukai describe $X$ as a member of $|\mathcal{L}^{\otimes 2} \otimes \mathcal{O}(2,3)|$ such that $X\cap Y$ is irreducible and $Y\in |\mathcal{L}|$ \cite[Table 3]{MM81classtable}. 

Let $\iota : X \to \mathcal{Q} $ be the inclusion map. We use $\zeta$ to denote the divisor class corresponding to $[\mathcal{L}]$. We also write $H_1$ and $H_2$ for the divisor classes corresponding to $[\mathcal{O}(1,0)]$ and $[\mathcal{O}(0,1)]$ on $S$, respectively. 
Then
\begin{align*}
-K_{\mathcal{Q}} = 3\zeta - f^*(K_S+\operatorname{det}(\mathcal{O}\oplus\mathcal{O}(-1,-1)^{\oplus 2}) &= 3\zeta - f^*(K_S-2H_1-2H_2)\\
&= 3\zeta + f^*(4H_1+4H_2).
\end{align*}
For the ease in notation, let us write $F_1 := \iota^*f^*H_1$, $F_2:=\iota^*f^*H_2$, $\xi = \iota^*\zeta$. Note that $\pic(X)= \zz[\xi]\oplus\zz[F_1]\oplus\zz[F_2]$. Using the adjunction formula, 
\begin{equation}  \label{anticanonicalfor3.2}
\begin{split}
-K_X &= \iota^*(3\zeta+f^*(4H_1+4H_2)-X) \\
&= \iota^*(3\zeta+f^*(4H_1+4H_2)-2\zeta-f^*(2H_1+3H_2))\\
&= \iota^*(\zeta+f^*(2H_1+H_2)) = \xi + 2F_1+F_2.
\end{split}
\end{equation}

\begin{lemma}   \label{Am:3.2}
   Let $X$ be a smooth Fano threefold in \textnumero 3.2. Let $G\subseteq \aut(X)$ be a finite group. Then $\am(X,G)$ is isomorphic to $\zz/2\zz$ or trivial.
\end{lemma}
\begin{proof}
    Let $G\subseteq \aut(X)$ be a finite group. Since $\autp(X)=0$ (Table \ref{table:primitive}), we may assume that $G$ is a subgroup of $\aut(\mathcal{Q};X)$ such that it acts trivially on $\pic(\mathcal{Q})= f^*\pic(S)\oplus \zz[\zeta]$. Then both $\iota$ and $f$ are $G$-equivariant. Therefore, Proposition \ref{KXlinearizable} implies that $2F_1$, $2F_2$, $\iota^*(-K_{\mathcal{Q}})$, and $-K_X$ are $G$-linearizable on $X$.
Moreover, since
\[
\iota^*(-K_{\mathcal{Q}}) + 2K_X = \xi + 2F_2,
\]
$\xi$ is $G$-linearizable. Consequently,
\[
\am(X,G) \subseteq \{0, \overline{F_1}, \overline{F_2}, \overline{F_1+F_2}\}. 
\]
However, from \eqref{anticanonicalfor3.2}, we have $\overline{F_2}=0$. Therefore, $\am(X,G)\subseteq \zz/2\zz$. 

We now use the construction from \cite[Example 5.13]{abban2025kstabilityfano3foldsworld}, which provides
an alternative description of $X$ together with a subgroup $\mathcal{G} \subset \aut(X)$. We briefly recall this construction and observe that the resulting action realizes the bound above.

Let $V:= \pp(\mathcal{O}_{\pp^1}(2)\oplus \mathcal{O}_{\pp^1}(2)\oplus \mathcal{O}_{\pp^1}(1)\oplus \mathcal{O}_{\pp^1}(1))$ be a toric fourfold, and let $f: V \to \pp^1_{t_1,t_2}$ denote the natural projection. Let $M$ be the tautological bundle on $V$ and let $F$ denote a fiber of $f$.

Consider the group $\mathcal{G} \subset \aut(V)$ generated by 
\begin{equation*}
\begin{aligned}
(t_1,t_2;x_1,x_2,x_3,x_4)
&\mapsto
(t_1,-t_2;x_1,-x_2,x_3,-x_4)\\
(t_1,t_2;x_1,x_2,x_3,x_4)
&\mapsto
(t_2,t_1;x_1,x_2,x_4,x_3),
\end{aligned}
\end{equation*}
where $(t_1,t_2;x_1,x_2,x_3,x_4)$ are Cox coordinates on $V$. The threefold $X$ belongs to the linear system  $|3M-4F|$ and is defined by 
\[
\begin{aligned}
    x_1x_3x_4 + 3x_2(x_3^2+x_4^2)+ 1007 t_1t_2x_1^3 - &301 t_1t_2 x_1x_2^2 +
        (t_1^2+t_2^2)(71x_1^2x_2 - 54 x_2^3) +\\
        &1111 x_1^2 (t_2x_3 + t_1x_4) - 99x_2^2 (t_2x_3+t_1x_4)=0.
\end{aligned}
\]
This equation is $\mathcal{G}$-invariant and $\mathcal{G}$ is a subgroup of $\aut(X)$. Moreover, the restriction of $f$, $\pi : X \to \pp^1$ is $\mathcal{G}$-equivariant. Note that $\mathcal{G}$ induces an action on $\pp^1$ generated by
\[
[t_1:t_2] \mapsto [t_1:-t_2], \quad [t_1:t_2] \mapsto [t_2:t_1],
\]
such that the induced projective representation does not lift to a linear representation. Thus, $\am(\pp^1,\mathcal{G})=\zz/2\zz$. Since $\pi$ is $\mathcal{G}$-equivariant, $\pi^*\mathcal{O}_{\pp^1}(1)$ is not $\mathcal{G}$-linearizable, so $\am(X,\mathcal{G})$ is nontrivial. Hence, $
\am(X,\mathcal{G})\cong \zz/2\zz$.
\end{proof}


\section{Double covers}  \label{divisorontoric}
In this section, we discuss the Amitsur groups of the following families of smooth Fano threefolds
\begin{center}
   \textnumero 2.2,\; \textnumero 2.8,\; \textnumero 2.18,\; \textnumero 3.1, \; \textnumero 3.4.
\end{center}

\subsection{\texorpdfstring{Double Covers of $\pp^2\times\pp^1$}{Double covers of P2P1}}
If $X$ is a member of \textnumero 2.2 or \textnumero 2.18, it can be described as a double cover of the Fano threefold $\pp^2\times \pp^1$ (\cite{MM81classtable}).
Suppose that $\pi : X \to \pp^2\times\pp^1$ is a double covering with smooth branch locus $B$. Let $H_1$ and $H_2$ denote the pullbacks to $\pp^2\times\pp^1$ of the hyperplanes on $\pp^2$ and $\pp^1$, respectively. Then $\pic(\pp^2\times\pp^1) = \zz[H_1]\oplus\zz[H_2]$. Further, if $B$ is ample then from \cite[Theorem 3.8]{MM86} 
\[
\pic(X)= \zz [\pi^*H_1]\oplus \zz[\pi^*H_2].
\]
In the diagram below, $\operatorname{pr}_1$ and $\operatorname{pr}_2$ denote the natural projections of $\pp^2\times\pp^1$ onto its first and second factors, respectively. The linear systems $|\pi^*H_1|$ and $|\pi^*H_2|$ define the morphisms $f_1: X \to \pp^2$ and $f_2: X \to \pp^1$, respectively, satisfying $f_i = \operatorname{pr}_i\circ\pi$ for $i=1,2$.
\[\begin{tikzcd}
	& X \\
	& {\mathbb{P}^2\times\mathbb{P}^1} \\
	{\mathbb{P}^2} && {\mathbb{P}^1}
	\arrow["\pi"', from=1-2, to=2-2]
	\arrow["{f_1}"', curve={height=12pt}, from=1-2, to=3-1]
	\arrow["{f_2}", curve={height=-12pt}, from=1-2, to=3-3]
	\arrow["{pr_2}", from=2-2, to=3-3]
	\arrow["{pr_1}", tail reversed, no head, from=3-1, to=2-2]
\end{tikzcd}\]

\begin{remark}  \label{double-cov-equivariant}
\normalfont If $X$ is in \textnumero 2.2 or \textnumero 2.18, $\autp(X)=0$ (Table \ref{table:primitive}). Let $G\subseteq \aut(X)$ be a finite group.
Then $\pi^*H_i$ is $G$-invariant and consequently, $f_i$ is $G$-equivariant for $i=1,2$. This implies that $\operatorname{pr}_i$ and $\pi$ are $G$-equivariant. It now suffices to consider only those groups $G$ that appear as the preimage of a subgroup $\widetilde{G}\subseteq \aut(\pp^2\times\pp^1;B)$. In addition, the pullbacks of canonical divisors along $G$-equivariant maps, 
\[
\pi^*\operatorname{pr}_1^*K_{\pp^2}=-3\pi^*H_1, \; \; \pi^*\operatorname{pr}_2^*K_{\pp^1} = -2\pi^*H_2
\]
are also $G$-linearizable on $X$ by Proposition \ref{KXlinearizable}.
\end{remark} 

Let us denote the image of $\pi^*H_i$ under the map $\pic(X)^G\to \am(X,G)$ by $\overline{\pi^*H_i}$, $i=1,2$. Then 
\[
\am(X,G) \subseteq \langle \overline{\pi^*H_1}, \overline{\pi^*H_2}\rangle,
\]
where the latter group is isomorphic to a subgroup of $\zz/6\zz$.
\begin{lemma}  \label{Am:2.2}
    Let $X$ be a smooth Fano threefold in family \textnumero $2.2$. Then for any finite group $G\subseteq \aut(X)$, we have $\am(X,G)=0$.
\end{lemma}
\begin{proof}
If $X$ is in family \textnumero 2.2, the branch locus $B$ is a divisor of bidegree $(4,2)$ on $\pp^2\times\pp^1$. The canonical divisor is 
\begin{align*}
    K_X &= \pi^*(-3H_1-2H_2 +(2H_1+H_2)) = -\pi^*H_1-\pi^*H_2
\end{align*}
by the Hurwitz formula (\cite[Theorem 3.8]{MM86}) and it is $G$-linearizable for any $G$ from Proposition \ref{KXlinearizable}. Since $\pi$ is $G$-equivariant from Remark \ref{double-cov-equivariant}, the ramification divisor $R:= K_X-\pi^*K_Y = \pi^*(2H_1+H_2)$ is $G$-linearizable on $X$.
Note that
\[
\pi^*H_1 = R + K_X, \hspace{0.7cm} \pi^*H_2 = -K_X-\pi^*H_1
\]
so $\pi^*H_1$, and consequently, $\pi^*H_2$ is $G$-linearizable. Thus, $\am(X,G)=0$. 
\end{proof}

\begin{lemma}   \label{Am:2.18}
    Let $X$ be a smooth Fano threefold in family \textnumero $2.18$. Then for any finite group $G\subseteq \aut(X)$, we have $\am(X,G)=0$.
\end{lemma}
\begin{proof}
If $X$ is in family \textnumero 2.18, the branch locus $B$ is a divisor of bidegree $(2,2)$ on $\pp^2\times\pp^1$. By the Hurwitz formula, the canonical divisor of $X$ is 
\[
K_X = -\pi^*(3H_1+2H_2)+\pi^*(H_1+H_2) = -2\pi^*H_1-\pi^*H_2,
\]
and it is $G$-linearizable.
Moreover, $\pi^*\operatorname{pr}_2^*(K_{\pp^1})$ and $\pi^*\operatorname{pr}_1^*(K_{\pp^2})$ are also $G$-linearizable (Remark \ref{double-cov-equivariant}).
Finally, since
\begin{align*}
\pi^*H_1 &= -2K_X + \pi^*\operatorname{pr}_2^*(K_{\pp^1})+ \pi^*\operatorname{pr}_1^*(K_{\pp^2}),
\end{align*}
it follows that $\pi^*H_1$ is $G$-linearizable. Thus, $\pi^*H_2$ is also $G$-linearizable, and we get $\am(X,G)=0$.
\end{proof}

\smallskip
As a result, we can determine the Amitsur groups of Fano threefolds in family \textnumero 3.4.
By \cite[Table 2]{MM81classtable}, a member $X'$ of family \textnumero 3.4 can be described as the blow-up of a Fano threefold $X$ in family \textnumero $2.18$ along a smooth fiber $C$ of the composition $X \to \mathbb{P}^2\times \mathbb{P}^1 \to \mathbb{P}^2$. Let $\phi: X'\to X$ be the blow-up map, and let $E$ denote the exceptional divisor. Then $\pic(X')= \phi^*\pic(X)\oplus \zz[E]$. 

\begin{coro}  \label{Am:3.4}
Let $X'$ be a smooth Fano threefold in family \textnumero $3.4$. Let $G\subseteq \aut(X')$ be a finite group. Then $\am(X',G)=0$.  
\end{coro}
\begin{proof}
    Since $\autp(X')=\autp(X)=0$, it suffices to assume that $G$ descends to a subgroup of $\aut(X;C)$.  
    Then the blow-up map $\phi$ is $G$-equivariant, and from Proposition \ref{imprimitivebcdp} and Lemma \ref{Am:2.18}, we get 
    $\am(X',G)=0$.
\end{proof}

\subsection{Double Cover of \textnumero 2.32} \label{Am:2.6(b)discussion}
Let $W$ denote a Fano threefold in \textnumero 2.32. Recall that $W$ can be described as a divisor of bidegree $(1,1)$ on $\pp^2\times\pp^2$ (\S~\ref{divisorsonP2P2}). A Fano threefold $X$ in \textnumero 2.6(b) is a double cover of $W$ with branch locus $B\in |-K_W|$ \cite{MM81classtable}. Let $\pi: X \to W$ denote the double cover, and let $\operatorname{pr}_i$ denote the natural projections from $\pp^2\times\pp^2$ onto its two factors, for $i=1,2$. Let $\iota : W \hookrightarrow \mathbb{P}^2 \times \mathbb{P}^2$ be the embedding of $W$, and set $H_i := \operatorname{pr}_i^*[\mathcal{O}_{\pp^2}(1)]$ for each $i$.
Then from \cite[Theorem 3.8]{MM86} 
\[
\pic(X) = \zz[\pi^*\iota^*H_1, \pi^*\iota^*H_2] \cong \pic(W). 
\]
Observe that the extremal contractions $f_i: X \to \pp^2$ are induced by linear systems $|\pi^*\iota^*H_i|$ and $f_i = \operatorname{pr}_i\circ\iota\circ\pi$.
\[\begin{tikzcd}
	& {\mathbb{P}^2} & \\
	X & W & {\mathbb{P}^2\times\mathbb{P}^2} \\
	& {\mathbb{P}^2}
	\arrow["{f_1}", from=2-1, to=1-2]
	\arrow["\pi", from=2-1, to=2-2]
	\arrow["{f_2}"', from=2-1, to=3-2]
	\arrow["\iota", from=2-2, to=2-3]
	\arrow["{\operatorname{pr}_1}"', from=2-3, to=1-2]
	\arrow["{\operatorname{pr}_2}", from=2-3, to=3-2]
\end{tikzcd}\]
Let $G\subseteq \aut(X)$ be a finite group. From Table \ref{table:primitive}, we recall that both $\autp(X)$ and $\autp(W)$ are isomorphic to $\zz/2\zz$. 

Let us write $\overline{\pi^*\iota^*H_i} \in \am(X,G)$ for the image of $\pi^*\iota^*H_i \in \pic(X)^G$.
By the Hurwitz formula, we have
\begin{equation}   \label{eqn:anticanonical-for-2.6(b)}
    K_X = \pi^*(K_W+\frac{B}{2}) = \pi^*\left( \frac{K_W}{2}\right) = -\pi^*(\iota^*H_1+\iota^*H_2), 
\end{equation}
and $K_X$ is $G$-linearizable by Proposition \ref{KXlinearizable}.
The following lemma establishes an upper bound on $\am(X,G)$.  
\begin{lemma}  \label{Am:2.6b}
    Let $X$ be a Fano threefold in \textnumero 2.6 (b). Let $G\subseteq \aut(X)$ be a finite group. Then $\am(X,G)\subseteq \zz/3\zz$.
\end{lemma}
\begin{proof}
Suppose $\autp(X,G)=0$. Then, for $i=1$, $2$, $\pi^*\iota^*H_i$ is $G$-invariant, and so the map $f_i$ is $G$-equivariant. Furthermore, since $\pi$ is a surjection and $\iota(W)$ is $G$-invariant, therefore, $\iota$ is $G$-equivariant. Thus, $\pi$ is also $G$-equivariant. We may now assume that $G$ appears as the preimage of a group $\widetilde{G}\subset \aut(\pp^2\times\pp^2;W)$ fixing $B$.  Since $f_i^*(-K_{\pp^2})$ is $G$-linearizable by Proposition \ref{KXlinearizable}, we get
    \[
    \am(X,G) \subseteq \langle \overline{\pi^*\iota^*H_1} \rangle,
    \]
where the latter is isomorphic to a subgroup of $\zz/3\zz$.
  If $\autp(X,G)=\zz/2\zz$, then $\pic(X)^G = \zz[\pi^*\iota^*(H_1+H_2)]$. From \eqref{eqn:anticanonical-for-2.6(b)}, we have $\am(X,G)=0$. 
\end{proof}

\subsection{Double Cover of \textnumero 2.35}
Now suppose that $Y$ denotes a smooth Fano threefold in family \textnumero 2.35, then it is the blow-up of $\pp^3$ in a point (\cite{MM81classtable}). Let $f : Y \to \pp^3$ be the blow-up morphism and $E$ denote the exceptional divisor of the blow-up. If $H$ is a general hyperplane on $\pp^3$, then $\pic(Y)=\zz[f^*H, E]$.

Let $X$ denote a member of family \textnumero 2.8, then it is a double cover of Y with a branch locus, $B \in |-K_Y|$. Let $\pi: X \to Y$ denote the double covering. There are two possibilities for $B$--- either (a) $B\cap E$ is smooth, or (b) $B \cap E$ is singular but reduced. 
From \cite[Theorem 3.8]{MM86}, 
\[
\pic(X)=\zz[\pi^*f^*H, \pi^*E].\] 

Observe that $h^0(\pi^*f^*H)= h^0(f^*H)=4$. In the diagram below, we write $\phi$ for the map induced by the linear system $|\pi^*f^*H|$. Note that $\phi=f\circ \pi$. 
\[\begin{tikzcd}
	X && Y \\
	& {\mathbb{P}^3}
	\arrow["\pi", from=1-1, to=1-3]
	\arrow["\phi"', from=1-1, to=2-2]
	\arrow["f", from=1-3, to=2-2]
\end{tikzcd}\] 
Let $G\subseteq \aut(X)$ be a finite group. Recall for the family \textnumero $2.8$, $\autp(X)=0$ (Table \ref{table:primitive}), so $\pi^*f^*H$ is $G$-invariant. This implies that the map $\phi$ is $G$-equivariant. 
\begin{lemma}  \label{Am:2.8}
    Let $X$ be a smooth Fano threefold in \textnumero $2.8$. Let $G\subseteq \aut(X)$ be a finite group. Then $\am(X,G)=0$. 
\end{lemma}
\begin{proof}
Observe that the induced $G$-action on $\pp^3$ lifts to an action on $Y$ since the center of the blow-up is $G$-invariant. Thus, $f$ is $G$-equivariant, and consequently, $\pi$ is also $G$-equivariant.  
Now from Theorem \ref{imprimitivebcdp} and $\am(Y,G)=0$ (Table \ref{table:toric}), we get $\am(X,G)=0$.
\end{proof}

\subsection{\texorpdfstring{Double Cover of $\pp^1\times\pp^1\times\pp^1$}{Double Cover of P1P1P1}}
Let $Y:=\pp^1\times\pp^1\times\pp^1$ be a smooth Fano threefold in \textnumero 3.27. If $X$ is in \textnumero $3.1$, it can be described as a double cover of $Y$ with branch locus $B$ a divisor of tridegree (2,2,2) \cite[Table 3]{MM81classtable} . 
 Let $\pi: X \to Y$ denote the double covering, and let $\operatorname{pr}_i: Y \to \pp^1$ denote the projection onto the $i$-th factor, for $i=1,2,3$. 
Set $H_i := \operatorname{pr}_i^*[\mathcal{O}_{\pp^1}(1)]$. Then $B\in |2H_1+2H_2+2H_3|$ and $\pic(Y)=\bigoplus_{i=1}^3 \zz[H_i]$. Since $B$ is ample, from \cite[Theorem 3.8]{MM86} we get
\[
\pic(X)= \bigoplus_{i=1}^3 \zz[\pi^*H_i] \cong \pic(Y).
\]
Also, $-K_X =\pi^*(H_1+H_2+H_3)$ by the Hurwitz formula. 
\smallskip 

Let us write $f_i$ for the map induced by the linear system $|\pi^*H_i|$. 
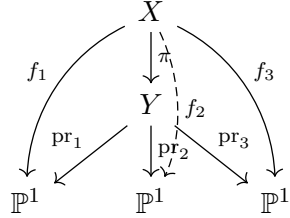
\[\begin{tikzcd}
	& X \\
	& Y \\
	{\mathbb{P}^1} & {\mathbb{P}^1} & {\mathbb{P}^1}
	\arrow["\pi", from=1-2, to=2-2]
	\arrow["{f_1}"', curve={height=18pt}, from=1-2, to=3-1]
	\arrow["{f_2}", curve={height=-12pt}, dashed, from=1-2, to=3-2]
	\arrow["{f_3}", curve={height=-18pt}, from=1-2, to=3-3]
	\arrow["{\operatorname{pr}_1}"', from=2-2, to=3-1]
	\arrow["{\operatorname{pr}_2}", from=2-2, to=3-2]
	\arrow["{\operatorname{pr}_3}", from=2-2, to=3-3]
\end{tikzcd}\]
where $f_i=\operatorname{pr}_i\circ\pi$. 

Let $G\subseteq \aut(X)$ be a finite group. Recall that both $\autp(X)$ and $ \autp(Y)$ are isomorphic to $S_3$ (Table \ref{table:primitive}).
Denote by $\overline{\pi^*H_i} \in \am(X,G)$ the image of $\pi^*H_i \in \pic(X)^G$.

\begin{remark} \hspace{-0.2cm}   \label{remark:for-3.1}
 \normalfont 
 \begin{enumerate}
     \item[(i)] Suppose $\autp(X,G)=0$. Then $f_i$ is $G$-equivariant for each $1\leq i \leq 3$, and so $\pi$ is $G$-equivariant. Thus, $f_i^*(-K_{\pp^1})$ and $K_X$ are both $G$-linearizable from Proposition~\ref{KXlinearizable}, in other words,
  \begin{equation}   \label{eqn:divisorrelationsfor3.1}
  \overline{2\pi^*H_i} =0, \;\; \overline{\pi^*(H_1+H_2+H_3)}=0, \hspace{0.8cm} \text{for each $i$.}
  \end{equation} 
Thus, we get 
\[
\am(X,G) \subseteq \langle \overline{\pi^*H_1}, \overline{\pi^*H_2}\rangle,
\]
where the latter group is isomorphic to a subgroup of $(\zz/2\zz)^2$.
\item[(ii)] Suppose $\autp(X,G)=\zz/2\zz$. Without loss of generality, we may assume that $G$ permutes the first two generators, $\pi^*H_1$ and $\pi^*H_2$, of $\pic(X)$. 
Therefore, the $G$-action swaps $f_1$ and $f_2$ up to an automorphism of $\pp^1$, and $f_3$ is $G$-equivariant, we get that $\pi$ is $G$-equivariant. Since 
\[
\pic(X)^G= \zz[\pi^*(H_1+H_2),\pi^*H_3],
\] we get 
\[
\overline{2\pi^*H_3} =0, \;\; \overline{\pi^*(H_1+H_2)}+\overline{\pi^*H_3}=0,
\]
and it follows that
\[
\am(X,G) \subseteq \langle \overline{\pi^*H_3}\rangle, 
\]
where the latter group is isomorphic to a subgroup of $\zz/2\zz$.
\item[(iii)] If $\zz/3\zz\subseteq \autp(X,G)$, then $\pic(X)^G=\zz[-K_X]$. Since $-K_X$ is $G$-linearizable, we get $\am(X,G)=0$. 
\end{enumerate}
\end{remark}

\smallskip
Let $([x_0:x_1],[y_0:y_1],[z_0:z_1])$ be coordinates on $Y= \pp^1\times\pp^1\times\pp^1$. 

\begin{lemma}   \label{Am:3.1}
    Let $X$ be a smooth Fano threefold in family \textnumero $3.1$. Let $G\subseteq \aut(X)$ be a finite group. Then $\am(X,G)$ is isomorphic to $(\zz/2\zz)^2$, $\zz/2\zz$, or trivial.
\end{lemma}
\begin{proof}
It remains to show that the upper bounds in cases (i) and (ii) of Remark \ref{remark:for-3.1} can be realized. In each case, we construct a finite group $\mathcal{G}\subset \aut(Y)$ and a smooth $\mathcal{G}$-invariant hypersurface $B\subset Y$ of tridegree $(2,2,2)$. Let $X$ be the double cover branched along $B$, and $G\subset \aut(X)$ denote the preimage of $\mathcal{G}$. Since $\pi$ is $G$-equivariant, using Proposition \ref{g-eq-divisors}, we get 
\[
\am(X,G)=\am(Y,G).
\]
Moreover, the surjection map $\am(Y,\mathcal{G})\twoheadrightarrow \am(Y,G)$ has kernel generated by the image of $B/2=H_1+H_2+H_3$ in $\am(Y,\mathcal{G})$. In both cases, we verify that $B/2$ is $\mathcal{G}$-linearizable, so $\am(Y,\mathcal{G})\cong \am(Y,G)$. Consequently,
\[
\am(X,G)\cong\am(Y,\mathcal{G}). 
\]
\smallskip
Consider the equation
\begin{equation}
    \begin{aligned}   \label{eqn:first-for3.1}
& a_0(x_0^2+x_1^2)(y_0^2+y_1^2)(z_0^2+z_1^2) + a_1x_0x_1y_0y_1(z_0^2+z_1^2) +\\
&a_2x_0x_1z_0z_1(y_0^2+y_1^2) + a_3y_0y_1z_0z_2(x_0^2+x_1^2) + \\
&a_4(x_0^2-x_1^2)(y_0^2-y_1^2)(z_0^2-z_1^2)= 0
   \end{aligned} 
\end{equation}
Let $\mathcal{G}_1\cong (\zz/2\zz)^3$ be a group generated by involutions of $Y$
\begin{align*}
       \sigma_1 : ([x_0:x_1],[y_0:y_1],[z_0:z_1]) &\mapsto ([x_0:-x_1],[y_0:-y_1],[z_0:-z_1]) \\
     \sigma_2 :   ([x_0:x_1],[y_0:y_1],[z_0:z_1])  &\mapsto ([x_1:x_0],[y_0:y_1],[z_1:z_0]) \\
     \sigma_3 : ([x_0:x_1],[y_0:y_1],[z_0:z_1])  &\mapsto ([x_0:x_1],[y_1:y_0],[z_1:z_0]). 
\end{align*}
Then the hypersurface in \eqref{eqn:first-for3.1} is $\mathcal{G}_1$-invariant, and it defines a smooth and irreducible hypersurface $B\subset Y$; for instance, for the choice of coefficients
$(a_0,\dots , a_4)= (7,4,2026,21,9)$ using Magma (\cite{magma}). Thus, $\mathcal{G}_1\subset \aut(Y;B)$.

Observe that $\autp(Y,\mathcal{G}_1)=0$. Since the induced $\mathcal{G}_1$-action on each factor of $Y$ does not lift to a group homomorphism to $\operatorname{GL}_2(\cc)$, none of $H_1$, $H_2$, and $H_3$ is $\mathcal{G}_1$-linearizable. The linear system $|H_1+H_2+H_3|$ induces the $\mathcal{G}_1$-equivariant Segre embedding $\nu:Y\to \pp^7$. Since the $\mathcal{G}_1$-action on $\pp^{7}$ lifts to $\operatorname{GL_8(\cc)}$, $\mathcal{O}_{\pp^7}(1)$ is $\mathcal{G}_1$-linearizable. Consequently, its pullback satisfies $\overline{H_1+H_2+H_3}=0$ in $\am(Y,\mathcal{G}_1)$. Combining with Remark \ref{remark:for-3.1} (i), we obtain
\[
\am(Y,G)\cong \am(Y,\mathcal{G}_1)\cong (\zz/2\zz)^2.
\]

Let $X$ denote the double cover of $Y$ branched along the hypersurface in \eqref{eqn:first-for3.1}, and let $G\subseteq \aut(X)$ denote the preimage of $\mathcal{G}_1$. Then, we get $\am(X,G)\cong (\zz/2\zz)^2$. 

\vspace{0.3cm}

Now consider the hypersurface $B \subset Y$ defined by
\begin{equation}   \label{eqn:second-for-3.1}
\begin{aligned}
&b_0(x_0^2y_0^2z_0^2+x_1^2y_1^2z_1^2) + b_1(x_0^2y_0^2z_1^2+x_1^2y_1^2z_0^2) + \\
&b_2(x_0^2y_1^2+x_1^2y_0^2)(z_0^2+z_1^2) + b_3z_0z_1(x_0^2y_1^2-x_1^2y_0^2) = 0.
\end{aligned}
\end{equation}
Let $\mathcal{G}_2 \cong (\zz/2\zz)^2$ be a subgroup of $\aut(Y)$ generated by
    \begin{align*}
       \sigma: ([x_0:x_1],[y_0:y_1],[z_0:z_1]) &\mapsto ([-y_0:y_1],[x_0:x_1],[z_0:-z_1])\\
      \tau :  ([x_0:x_1],[y_0:y_1],[z_0:z_1])  &\mapsto ([y_1:-y_0],[x_1:x_0],[z_1:z_0]).
    \end{align*}
Under the Segre embedding $\nu: \pp^1_{x_0,x_1}\times\pp^1_{y_0,y_1} \hookrightarrow \pp^3$, these induce the involutions
\begin{align*}
   \sigma': [w_0:w_1:w_2:w_3] &\mapsto [-w_0:w_2:-w_1:w_3]\\
    \tau':[w_0:w_1:w_2:w_3] &\mapsto [w_3:-w_1:w_2:-w_0], 
\end{align*}
where $[w_0:w_1:w_2:w_3]$ are coordinates on $\pp^3$. Observe that since $\mathcal{G}_2$-action swaps the generators $H_1$ and $H_2$ of $Y$, we have $\autp(Y,\mathcal{G}_2)=\zz/2\zz$. 
The divisor $H_3$ is indeed not $\mathcal{G}_2$-linearizable since the induced $\mathcal{G}_2$-action on the third factor $\pp^1_{z_0,z_1}$ does not lift to a group homomorphism to $\operatorname{GL}_2(\cc)$. Moreover, $H_1+H_2$ is also not $\mathcal{G}_2$-linearizable, since $\sigma'$ and $\tau'$ commute in $\operatorname{PGL}_4(\cc)$, but not in $\operatorname{GL}_4(\cc)$. 
Finally, the Segre embedding induced by $|H_1+H_2+H_3|$, $Y \to \pp^7$ is $\mathcal{G}_2$-equivariant such that $\mathcal{O}_{\pp^7}(1)$ is $\mathcal{G}_2$-linearizable. Hence its pullback $B/2$ to $Y$ is also $\mathcal{G}_2$-linearizable. Applying Remark \ref{remark:for-3.1} (ii), we obtain
\[\am(Y,G)=\am(Y,\mathcal{G}_2)\cong \zz/2\zz.\] 

Note that the hypersurface defined in \eqref{eqn:second-for-3.1} is $\mathcal{G}_2$-invariant, smooth, and irreducible. The latter two assertions can be verified for the choice of coefficients $(b_0, b_1, b_2,b_3)=(21,9,2026,4)$ using Magma (\cite{magma}). If $X$ denotes the double cover of $Y$ branched along this hypersurface, and $G\subseteq \aut(X)$ is the preimage of $\mathcal{G}_2$. Then $\am(X,G)\cong \zz/2\zz$. 

\end{proof}

\newpage
\section{Tables for Amitsur groups}  \label{sec:tables}
In the following tables, column 1 lists the Mori–Mukai names of the deformation families of smooth Fano threefolds. Column 2 lists the possible groups $\autp(X,G)$, i.e., the images of finite groups $G\subseteq \aut(X)$ in $\aut(\pic(X))$. For each entry in column 2, column 3 records the largest possible Amitsur group $\am(X,G)$. Column 4 of Table \ref{table:primitive} lists the references where $\am(X,G)$ has been classified for the corresponding family.

\small{\begin{longtable}   {|p{1.9cm}|p{2.0cm}|p{3.0cm}|p{3.5cm}|}
 \caption{Primitive Fano threefolds} \label{table:primitive}\\ 
    \hline
     MM \textnumero & $\autp(X,G)$  & $\am(X,G)$ &  Reference\\
\hline
     1.1,\dots , 1.10    & 0    & 0 & Lemma \ref{amfor_index=rank=1}\\
\hline
     1.11    & 0    & $0$ & Lemma \ref{Am:1.11} \\
\hline
     1.12  & 0    & $\zz/2\zz$ & Lemma \ref{Am:1.12}  \\
\hline
    1.13  & 0  &    $0$ & Lemma \ref{Am:1.13,1.16} \\
\hline
    1.14  & 0  &    $\zz/2\zz$ & Lemma \ref{Am:1.14} \\
\hline
   1.15  &  0 &    $0$ & Lemma \ref{Am:1.15}  \\
\hline
   1.16  & 0  &    $0$ & Lemma \ref{Am:1.13,1.16} \\
\hline
   1.17  &  0 &    $\zz/4\zz$ & Example \ref{amforP3}\\
\hline
 2.2  & $0$  &   0 & Lemma \ref{Am:2.2}\\
\hline
 2.6 (a)  &  $\zz/2\zz$ &   $0$ & Lemma \ref{Am:2.6a} (see\\
         &    $0$    &          $\zz/3\zz$  &      also Lemma \ref{Am:2.6b})\\
\hline
 2.8  & 0  &    $0$ & Lemma \ref{Am:2.8} \\
\hline
 2.18  &  0 &   $0$ &  Lemma \ref{Am:2.18} \\
\hline
 2.24  & $0$  &  $\zz/3\zz$ & Lemma \ref{Am:2.24} \\
 \hline
 2.32  & $\zz/2\zz$  &  $0$ & Lemma \ref{Am:2.32} \\ 
      &   $0$    &      $\zz/3\zz$ &  \\
\hline
2.34  & $0$  &  $0$ & Table \ref{table:toric}\\
\hline
2.35  & $0$  &  $0$ & Table \ref{table:toric} \\
\hline 
 2.36  & $0$  &  $0$ &  Table \ref{table:toric}\\
\hline
 3.1    &  $0$     &  $(\zz/2\zz)^2$ & Lemma \ref{Am:3.1}\\
        &  $\zz/2\zz$  & $\zz/2\zz$  & \\
       &  $\zz/3\zz$  & $0$  &  \\
       & $S_3$  &   $0$ &  \\
\hline
 3.2  & $0$  & $\zz/2\zz$ & Lemma \ref{Am:3.2}\\
\hline
 3.27  & $S_3$  &  $\zz/2\zz$ &  Table \ref{table:toric}\\
 & $\zz/3\zz$  & $\zz/2\zz$  &  \\
 & $\zz/2\zz$ &  $(\zz/2\zz)^2$ &   \\
 & $0$ &  $(\zz/2\zz)^3$   &   \\
\hline
 3.31  & $\zz/2\zz$  & $0$ & Table \ref{table:toric}\\
 &     $0$  &   $\zz/2\zz$  &  \\
\hline
\end{longtable}}

\pagebreak
\small{\begin{longtable}   {|p{1.5cm}|p{2.0cm}|p{3.0cm}|} 
 \caption{Toric Fano threefolds} \label{table:toric}\\
    \hline
     MM \textnumero & $\autp(X,G)$  & $\am(X,G)$ \\
\hline
   1.17  &  0 &    $\zz/4\zz$\\
\hline
 2.33  & $0$  &  $\zz/2\zz$ \\ 
\hline
 2.34  & $0$  &  $\zz/6\zz$\\  
\hline
 2.35  & $0$  &  $0$ \\
\hline
 2.36  & $0$  &  $0$ \\
\hline
3.25  &  $0$ &  $\zz/2\zz$  \\
      &  $\zz/2\zz$   &   $\zz/4\zz$\\
\hline
3.26  &  $0$ &   $0$   \\
\hline
 3.27  & $S_3$  &  $\zz/2\zz$ \\
 & $\zz/3\zz$  & $\zz/2\zz$   \\
 & $\zz/2\zz$ &  $(\zz/2\zz)^2$   \\
 & $0$ &  $(\zz/2\zz)^3$     \\
\hline
 3.28  & $0$  &  $\zz/2\zz$  \\
\hline
 3.29  & $0$  &  $0$  \\
\hline
 3.30  & $0$  & $0$ \\
\hline
 3.31  & $\zz/2\zz$  & $0$ \\
 &     $0$  &   $\zz/2\zz$   \\
\hline
 4.9 &  $0$ &  $0$  \\
\hline
 4.10  & $\zz/2\zz$ & $\zz/2\zz$  \\
    &  $0$ & $\zz/2\zz$  \\
\hline
 4.12  & $\zz/2\zz$  &  $\zz/2\zz$ \\
 &  $0$  &    $0$    \\
\hline
 4.11  & $0$  & $0$ \\
\hline
 5.2  &  $\zz/2\zz$   & $\zz/2\zz$ \\
   &   $0$  &   $0$   \\
\hline
5.3  & 0 & $\zz/2\zz$ \\
 &   $\zz/2\zz$ & $\zz/2\zz$  \\
 &    $\zz/2\zz$ & $\zz/2\zz$ \\
 &   $\zz/2\zz$ & $(\zz/2\zz)^3$  \\
 &  $\zz/3\zz$ & $\zz/6\zz$  \\
 &    $(\zz/2\zz)^2 $ & $(\zz/2\zz)^2$ \\
 &  $S_3$ & $\zz/2\zz$  \\
 &   $S_3$ & $\zz/6\zz$   \\
 &    $\zz/6\zz$ & $\zz/2\zz$   \\
 &   $D_6$ & $\zz/2\zz$    \\
\hline
\end{longtable}}

\bibliographystyle{alphaurl} 
\bibliography{bibis} 

\end{document}